\documentclass[12pt]{amsart}
\usepackage{bbm}
\usepackage{amscd,amssymb,hyperref}
\usepackage[PostScript=dvips]{diagrams}
\usepackage{amsfonts}
\usepackage[all]{xy}

\newtheorem{thm}{Theorem}[section]
\newtheorem{lem}[thm]{Lemma}
\newtheorem{cor}[thm]{Corollary}
\newtheorem{prop}[thm]{Proposition}

\newtheorem{dfn}[thm]{Definition}

\newtheorem{rem}[thm]{Remark}
\newtheorem{example}[thm]{Example}
\newtheorem{question}[thm]{Question}

\def\square{\vbox{
      \hrule height 0.4pt
      \hbox{\vrule width 0.4pt height 5.5pt \kern 5.5pt \vrule width 0.4pt}
      \hrule height 0.4pt}}

\def\id{\mathrm{id}}

\def\Ker{\mathrm{K er}}
\def\ch\mathrm{c h}

\long\def\symbolfootnote[#1]#2{\begingroup%
\def\thefootnote{\fnsymbol{footnote}}\footnote[#1]{#2}\endgroup}

\newcommand{\Z}{\mathbb{Z}}

\newcommand{\calZ}{\ensuremath{\mathcal{Z}}}

\newcommand{\calA}{\ensuremath{\mathcal{A}}}

\newcommand{\calG}{\ensuremath{\mathcal{G}}}

\newcommand{\Brun}{\mathrm{Brun}}
\newcommand{\LPBn}{\mathrm{L^P(\Brun_n)}}

\numberwithin{equation}{section}

\newcommand{\auths}[1]{\textrm{#1},}
\newcommand{\artTitle}[1]{\textsl{#1},}
\newcommand{\jTitle}[1]{\textrm{#1}}
\newcommand{\Vol}[1]{\textbf{#1}}
\newcommand{\Year}[1]{\textrm{(#1)}}
\newcommand{\Pages}[1]{\textrm{#1}}

\title{Brunnian Braids and Lie Algebras }

\author{J.Y. Li }
\address{Department of Mathematics and Physics, Shijiazhuang Tiedao University 050000, China}
\email{yanjinglee@163.com}

\author{V. V. Vershinin }
\address{D\'epartement des Sciences Math\'ematiques,
                               Universit\'e Montpellier II,
Place Eug\`ene Bataillon,
34095 Montpellier cedex 5, France} \email{ vershini@math.univ-montp2.fr}
\address{Sobolev Institute of Mathematics, Novosibirsk 630090,
Russia }
\email{ versh@math.nsc.ru}
\address{Laboratory of Quantum Topology, Chelyabinsk State University, Brat'ev
Kashirinykh street 129, Chelyabinsk 454001, Russia}

\author{J. Wu }
\address{Department of Mathematics, National University of Singapore, 2 Science Drive 2
Singapore 117542} \email{matwuj@nus.edu.sg}
\urladdr{www.math.nus.edu.sg/\~{}matwujie}


\begin{document}

\begin{abstract}
Brunnian braids have interesting relations with homotopy groups 
of spheres. In this work, we study the graded Lie algebra of the
descending central series related to Brunnian subgroup of the pure
braid group. A presentation of this Lie algebra is obtained.
\end{abstract}

\maketitle


\section{Introduction}
The pure braid group $P_n$ (of a disc) can be given by the following
presentation:

\noindent generators:  $a_{i,j}$, $1\leq i<j\leq n $,

\noindent the defining relations (\emph{Burau relations}
(\cite{Bu1}, \cite{Ve10})):
\begin{equation}
\begin{cases}
a_{i,j}a_{k,l}=a_{k,l}a_{i,j}
\ \text {for} \ i<j<k<l \ \text {and} \ i<k<l<j, \\
a_{i,j}a_{i,k}a_{j,k}=a_{i,k}a_{j,k}a_{i,j} \ \text {for} \
i<j<k, \\
a_{i,k}a_{j,k}a_{i,j}=a_{j,k}a_{i,j}a_{i,k} \ \text
{for} \ i<j<k, \\
a_{i,k}a_{j,k}a_{j,l}a_{j,k}^{-1}=a_{j,k}a_{j,l}a_{j,k}^{-1}a_{i,k}
\ \text {for} \ i<j<k<l.\\
\end{cases}
\label{eq:burau}
\end{equation}
 A geometric braid is called Brunnian if (1) it is a pure braid and (2)
   it becomes trival braid by removing any of its strands. Since the composition of any two Brunnian braids is still Brunnian,
   the set of Brunnian braids is a  subgroup of the braid group which is denoted by $\Brun_n$. By a direct geometric observation  $\Brun_n$ is the
normal subgroup of $P_n$, it is  generated by the iterated commutators
$$
[[[a_{1,2},a_{i_2,3}],a_{i_3,4}],\ldots,a_{i_{n-1},n}]
$$
for $1\leq i_t\leq t$ and $2\leq t\leq n-1$, where the commutator
$[a,b]$ is defined by $[a,b]=a^{-1}b^{-1}ab$ for $a,b\in G$  ~\cite{LW1}.

Brunnian braids have connections with homotopy theory as
described in ~\cite{BCWW}, ~\cite{LiWu}  and \cite{BMVW}.

We remind that for a group $G$ the descending central series
\begin{equation*}
G =\Gamma_1  \geq \Gamma_2 \geq \dots  \geq \Gamma_i \geq
\Gamma_{i+1} \geq \dots .
\end{equation*}
\noindent
is defined by the formulae
\begin{equation*}
\Gamma_1 = G, \ \ \Gamma_{i+1} =[\Gamma_{i}, G].
\end{equation*}
The descending central series of a discrete group $G$ gives rise to the
associated graded Lie algebra (over $\Z$) $L(G)$ 
\begin{equation*}
L_i(G)= \Gamma_i(G)/\Gamma_{i+1}(G).
\end{equation*}

The descending central series and the associated Lie algebras of the pure braid groups have been studied in particular in the works~\cite{CKX,GM-P,KaaV, K,Stanford}. It is also an ingredient in the study  of Vassiliev invariants of braids.
The associated graded algebra of  the Vassiliev filtration for pure braid group ring coincides with the associated algebra of the filtration by the powers of augmentation ideal of the group ring of pure braids. The latter by the Quillen's theorem~\cite{Quillen} is connected with the universal enveloping algebra of the associated Lie algebra of the descending central series of the pure braid group.

In this work, we consider the restriction $\{\Gamma_q(P_n)\cap\Brun_n\}$ of the descending central series of $P_n$ to $\Brun_n$. This gives a relative Lie algebra
\begin{equation}
\LPBn
=\bigoplus_{q=1}^\infty (\Gamma_q(P_n)\cap\Brun_n)/(\Gamma_{q+1}(P_n)\cap\Brun_n),
\label{LPBn}
\end{equation}
which is a two-sided Lie ideal of $L(P_n)$. The purpose of this article is to study the Lie algebra $\LPBn$.


We remark that the group $\Brun_n$ is a free group of infinite rank for $n\geq 4$ and so the associated Lie algebra $L(\Brun_n)$ is an infinitely generated free Lie algebra for $n\geq4$. The relative Lie algebra $\LPBn$ has better features, in particular  it is of finite type (in graded sense).

The main aim of the paper is to look at Brunnian braids at the level of Lie algebras.
Proposotions \ref{proposition1},  \ref{Proposition2} and \ref{proposition3}
as well as some subsequent statements
are the Lie algebra analogues of the corresponding facts for Brunnian groups.

\smallskip

\section{Lie algebra $\LPBn$}

A presentation of the Lie algebra $L(P_n)$ for the pure braid group
can be described as follows \cite{K}. It is the quotient of the free
Lie algebra $L[A_{i,j}| \, 1 \leq i < j \leq n]$ generated by
elements $a_{i,j}$ with $1 \leq i < j \leq n$ modulo the
``infinitesimal braid relations" or ``horizontal $4T$ relations"
given by the following three relations:

\begin{equation}
\begin{cases}
 [A_{i,j}, A_{s,t}] = 0, \  \text{if} \ \{i,j\} \cap \{s,t\} = \phi, \\
 [A_{i,j}, A_{i,k} + A_{j,k}] = 0, \  \text{if} \ i<j<k , \\
  [A_{i,k}, A_{i,j} + A_{j,k}] = 0, \ \text{if} \ i<j<k. \\
  \end{cases}
  \label{eq:kohno}
\end{equation}
Where $A_{i,j}$ is the projection of the $a_{i,j}$ to $L(P_n)$.

Let $G$ be a group with filtration $w$ (in the sense of Serre 
\cite[p.~7]{Serr}).
The fact that $\LPBn$ as defined in (\ref{LPBn}) 
 is a   Lie algebra  is a corollary of the following evident statement.
 \begin{prop}\label{proposition0}
For any subgroup $H$ of $G$ the restriction on $H$ of filtration $w$
defines a filtration on $H$.
 \hfill $\Box$
\end{prop}
We define a filtration $w_B$ on $\Brun_n$ by the formula:
$$w_B(b)= \inf\{p \ | \ b\in \Gamma_p\}.
$$
\begin{prop}\label{proposition1}
$\LPBn$ is a Lie algebra defined by the filtration $w_B$, it is a two-sided Lie ideal in $L(P_n)$. 
\end{prop}
\begin{proof}
The last statement follows from the fact that $\Brun_n$ is a normal subgroup of $P_n$.
\end{proof}
We call $\LPBn$ {\it relative Lie
algebra associated with Brunnian subgroup} of the pure braid group.

The removing-strand operation on braids induces an operation
$$
d_k\colon L(P_n)\longrightarrow L(P_{n-1})
$$
formulated by
\begin{equation}
d_k(A_{i,j})=\left\{
\begin{array}{lcl}
A_{i,j}&\textrm{ if } & i<j<k\\
0&\textrm{ if }& k=j\\
A_{i,j-1}&\textrm{ if }& i<k<j\\
0&\textrm{ if }& k=i\\
A_{i-1,j-1}&\textrm{ if }& k<i<j.\\
\end{array}\right.
\label{face-dk}
\end{equation}

A sequence of sets $S=\{S_n\}_{n\geq 0}$ is called a bi-$\Delta$-set
if there are faces $d_j:S_n\rightarrow S_{n-1}$ and co-faces
$d^j:S_{n-1}\rightarrow S_n $ for $0\leq j\leq n$ such that the
following identies hold:

\begin{enumerate}
\item $d_jd_i=d_id_{j+1}$ for $j\geq i$;
\item $d^jd^i=d^{i+1}d^j$ for $j\leq i$;
\item $d_jd^i=\left\{
\begin{array}{lcl}
d^{i-1}d_j&\textrm{ if }& j<i,\\
\id&\textrm{ if }& j=i,\\
d^id_{j-1}&\textrm{ if }& j>i.\\
\end{array}
\right.$
\end{enumerate}
In other words, $S$ is a bi-$\Delta$- and co-$\Delta$-set such that
relation (1)-(3) holds. Moreover a sequence of groups $\mathcal {G}$
is called a bi-$\Delta$-group if $\mathcal {G}$ is a bi-$\Delta$-set
such that all faces and co-faces are group homomorphism.

Let $\mathbb{P}_n=P_{n+1}.$ According to~\cite[Example 1.2.8]{Wu4},
the sequence of groups $\mathbb{P}=\{\mathbb{P}_{n}\}_{n\geq 0}$
with faces relabeled as $\{\mathbbm{d}_0,\mathbbm{d}_1,\ldots\}$ and
co-faces relabeled as $\{\mathbbm{d}^0,\mathbbm{d}^1,\ldots\}$ forms
a bi-$\Delta$-group structure. Where the face operation
$\mathbbm{d}_i:\mathbb{P}_n\rightarrow
\mathbb{P}_{n-1}=d_{i+1}:P_{n+1}\rightarrow P_n$ is obtained by
deleting the $i+1$st string,  the co-face operation
$\mathbbm{d}^i:\mathbb{P}_n\rightarrow \mathbb{P}_{n+1}$ is obtained
by adding a trivial $i+1$st string in front of the other strings
$(i=0,1,2,\cdots, n)$.

%
%

\begin{prop}\label{Proposition2}
 The relative Lie algebra $\LPBn$ is the
 Lie  subalgebra  $\bigcap_{i=1}^n\ker(d_i:L(P_n)\rightarrow
L(P_{n-1}))$.
\end{prop}
\begin{proof}
The assertion follows from \cite[Proposition 1.2.10]{Wu4}.
\end{proof}

Our next step is to determine a set of generators for the Lie algebra $\LPBn$. The following fact is a Lie algebra analogue of the theorem proved by A.~A.~Markov \cite{Mar2} for the pure braid group.
Also it follows from Theorem~3.1  in  \cite{fr} or
Lem\-ma~3.1.1 in \cite{Ih}.
\begin{prop}\label{proposition3}
The kernel of the homomorphism $d_n:L(P_n) \to L(P_{n-1})$ is a free
Lie algebra, generated by the free generators
 $A_{i,n}$, for $1\le i\le n-1$.
\begin{equation*}
 \Ker (d_n:L(P_n)\to L(P_{n-1}))= L[A_{1,n} , \ldots, A_{n-1,n}].
\end{equation*}
\hfill
$\Box$
\end{prop}

For a set $Z$, let $L[Z]$ denote
the free Lie algebra freely generated by $Z$.
Let $X$ and $Y$ be nonempty (possibly infinite) sets with $X\cap Y=\emptyset$,
$X\cup Y=Z$. We are interested to study the kernel of Lie homomorphism
$$
\pi\colon L[Z]\longrightarrow L[Y]
$$
$\pi$
such that
$\pi(x)=0$ for $x\in X$ and $\pi(y)=y$ for $y\in Y$.
The following lemma
is not new: for the case of Lie algebras over a field and when $X$ consists of one element this is Lemma~2.6.2 in \cite{BK}. For completeness of our exposition we are giving our proof here.
\begin{lem}\label{lemma4}
The kernel of $\pi$ is a free Lie algebra, generated by the
following family of free generators:
\begin{equation}\label{equation-lemma4}
x,  [\cdots [x, y_1], \ldots, y_t]
\end{equation}
for $x\in X, y_i\in Y$ for $1\leq i\leq t$.
\end{lem}
\begin{proof}
Observe that the kernel $\Ker(\pi)$ is the two-sided ideal generated by the elements $x\in X$. Let us prove that it is generated as a Lie algebra by the elements (\ref{equation-lemma4}). We prove this by induction on the length of monomials $M$, sums of which give the ideal. For the length 1 and 2 one can see this directly. Let the length of $M$ be at least 3: $M=[A,B]$ such that $A$ contains some $x\in X$. We may assume that
the length of $A$ is at least 2. If not, then $A=x$ for some $x\in X$ and $B=[B_1,B_2]$ and $M= [x, [B_1, B_2]]=
[[B_2,x], B_1] + [[x, B_1], B_2]$ and it is reduced to the case when in $M$ element
$A$ contains some $x\in X$ with its length at least 2. Since $A\in \Ker(\pi)$ with its length strictly less than that of $M$, it is a linear combination of the products $[A_1,A_2]$ of the generators (\ref{equation-lemma4})  with both $A_1$ and $A_2$ containing some (possibly different) element(s) in $X$ by induction. Consider the equality
 $$ [[A_1, A_2], B]=
-[[A_2,B], A_1] - [[B, A_1], A_2].$$
The elements $[A_2,B]$ and $[B, A_1]$ have length strictly less than that of $M$
and by induction they are given by linear combinations of  products of the generators (\ref{equation-lemma4}). Thus $M=[A,B]$ is a linear combination of products of the generators (\ref{equation-lemma4}).

Let us prove now that the elements (\ref{equation-lemma4}) freely generate
our ideal $\Ker(\pi)$. Let us define a free Lie algebra (over $\Z$)
that is freely generated by  formal elements \{$C(x)$,
$C(x,y_1,y_2,\ldots,y_t)$\}, $x\in X$, $y_1,\ldots,y_t\in Y$ with
$t\geq 1$, which are in one-to-one correspondence with the elements
(\ref{equation-lemma4})
$$F=L[C(x), \ C(x,y_1,y_2,\ldots,y_t) \ |  x\in X,\  y_1,\ldots,y_t\in Y,\ t\geq 1].$$
Let us define an action of the free Lie algebra $L[X\sqcup Y]$ on
$F$ by the formulae which mimic the action of $L[X\sqcup Y]$ on the
elements (\ref{equation-lemma4}):
 \begin{equation}
\begin{cases}
  [C(x,y_1,y_2,\ldots,y_t), y] = C(x,y_1,y_2,\ldots,y_t,y),  y\in Y\\
 [C(x,y_1,y_2,\ldots,y_t), x']=[C(x,y_1,y_2,\ldots,y_t), C(x')], x'\in X,
    \end{cases}\label{action-lemma4}
    \end{equation}
where, for $t=0$, $C(x,y_1,y_2,\ldots,y_t)=C(x)$.
Let us denote the generators (\ref{equation-lemma4}) of our ideal by $B(x)$, $B(x,y_1,y_2,\ldots,y_t)$, $x\in X$, $y_1,\ldots,y_t\in Y$ with $t\geq 1$.  We claim that the element $B(x,y_1,y_2,\ldots,y_t)$ acts on $F$ the same way as the inner derivation
by $C(x,y_1,y_2,\ldots,y_t)$:
\begin{equation*}\begin{cases}
[C(x',y'_1,\ldots,y'_{t'}), B(x,y_1,\ldots,y_t)]=[C(x',y'_1,\ldots,y'_{t'}), C(x,y_1,\ldots,y_t)].\\
\end{cases}
\end{equation*}

The proof is by induction on the length $t$ of $B(x,y_1,\ldots,y_t)$.  For the length $t=0$ it follows from the
 definition of the action. Let it be proved for  the lengths less than $t$ with $t>0$. Let $D=C(x',y'_1,\ldots,y'_{t'})$, $C=C(x,y_1,\ldots,y_t)$, $B'=B(x,y_1,\ldots,y_{t-1})$ and $C'=C(x,y_1,\ldots,y_{t-1})$.
$$
\begin{array}{rcl}
 [D,B(x,y_1,\ldots,y_t)]&=&[D, [B', y_t] ]\\
&=& [[D, B'], y_t]  - [[D,  y_t],B']\\
& = & [[D, C'], y_t]  - [C(x',y'_1,\ldots,y'_{t'}, y_t),B']\\
&&\quad \textrm{   ( by induction) } \\
&=&[[D, C'], y_t]  -  [C(x',y'_1,\ldots,y'_{t'}, y_t),C']\\
&&\quad \textrm{  (by induction)}\\
&=&  [[D, C'], y_t]  -
   [[D,  y_t],C']\\
&=&[D, [C', y_t] ]\\
&=&[D,C].\\
\end{array}
$$
The induction is finished.

  Let $D(F)$ be the Lie algebra of all derivations of the algebra $F$, homomorphism
  $\chi: F\to D(F)$ is defined by inner derivations.
 We define homomorphism $\phi: F\to  L[X\sqcup Y]$  by the formulae
  $$\phi(C(x))=B(x)=x\quad \textrm{ and}$$
$$ \phi(C(x,y_1,\ldots,y_t))= B(x,y_1,\ldots,y_t)=[\cdots[x,y_1],\ldots,y_t]$$
   and homomorphism  $\delta: L[X\sqcup Y]
  \to D(F)$ is defined by the action (\ref{action-lemma4}). There is a commutative diagram:
  \[\xymatrix{ F  \ar[rr]^{\phi} \ar[dr]_{\chi}
&& {L[X\sqcup Y]} \ar[dl]^{\delta} \\  & D(F). }  \]
The
homomorphism $\chi$ is a monomorphism as free Lie algebras with more
than 2 generators have trivial center \cite[Exercice 3), \S 3,
p.~79]{Bo}. So $\phi$ is also a monomorphism and hence  it is an
isomorphism on the ideal, generated by $B(x,y_1,\ldots,y_t)$.

\end{proof}

\begin{prop}\label{proposition5}
The intersection of the kernels of the homomorphisms $d_n$ and $d_{k}$, $k\not=n$,
is a free Lie algebra, generated by the following infinite family of free generators:
\begin{equation}
A_{k,n},  [\cdots [A_{k,n}, A_{j_1,n}], \ldots, A_{j_m, n}]
\label{genK}
\end{equation}
for
$ j_i\not= k, n; \ j_i\le n-1$; $i\le m$; $m\ge 1$:
\begin{multline}
 \Ker (d_n) \cap \Ker (d_{k}) = \\
 L[A_{k,n}, \,  [\cdots [A_{k,n}, A_{j_1,n}], \ldots, A_{j_m, n} ] \ | \
j_i\not= k, n; \ j_i\le n-1, i\le m; \ m\ge 1].
\label{obrazu}
\end{multline}
\end{prop}
\begin{proof}
Let us suppose for simplicity that $k=n-1$ and denote $A_{i,n}$ by $B_i$. Then the algebra from Proposition~\ref{proposition3} is the following free Lie algebra
$
 L[B_{1} , \ldots, B_{k}],
 $
  and the homomorphism $d_k$ can be expressed by the formulae
  \begin{equation*}
\begin{cases}
  B_1\mapsto B_1,\\
 \  \ \cdots,\\
  B_{k-1}\mapsto B_{k-1},\\
  B_k\mapsto 0.
    \end{cases}
    \end{equation*}
The assertion follows from Lemma~\ref{lemma4}.
\end{proof}
Another set of free generators of $ \Ker (d_n) \cap \Ker (d_{k})$ can be obtained
using Hall bases \cite{Bo}, \cite{Ha}. We remind the definition. We suppose that
all Lie monomials on $B_1,\ldots, B_k$ are ordered lexicographically.

Lie monomials $B_1,\ldots, B_k$ are the {\it standard} monomials of
degree $1$. If we have defined standard monomials of degrees $1, \ldots
, n - 1$,  then $[u, v]$ is a {\it standard} monomial if both of
the following conditions hold:

(1) $u$ and $v$ are standard monomials and $u>v$.

(2) If $u=[x, y]$ is the form of the standard monomial $u$, then $v\ge y$.

\noindent
Standard monomials form  the {\it Hall basis}  of a free Lie algebra (also over $\Z$).
Examples of standard monomials
are the products of the type:
\begin{equation}
[\cdots [B_{j_1}, B_{j_2}], B_{j_3}],\ldots, B_{j_t}], \ j_1> j_2\le
j_3\le \cdots\le j_t. \label{st}
\end{equation}
\begin{prop}\label{proposition6}
The intersection $\Ker(d_n)\cap \Ker(d_k)$, $k\not=n$,
is a free Lie algebra, generated by the standard monomials on $A_{i,n}$ where
the letter $A_{k,n}$ has only one enter. In other words the free generators are
standard monomials which are products of monomials of type (\ref{st}) where only one
such monomial contains one copy of $A_{k, n}$.
\end{prop}
\begin{proof}
We apply the procedure of constructing of a set of free generators
for a sub Lie algebra which was used by Shirshov \cite{S} and Witt \cite{Wi} in their proofs that a
Lie subalgebra of
a free Lie algebra is free.
\end{proof}

Lemma~\ref{lemma4} is useful for having an algorithm to recursively determine a set of free generators for $\LPBn$. A \textit{Lie monomial} $W$ on the letters $A_{1,n},A_{2,n},\ldots,A_{n-1,n}$ means $W=A_{i,n}$ for some $1\leq i\leq n-1$ or a Lie bracket $W=[A_{j_1,n}, A_{j_2,n},\ldots, A_{j_t,n}]$ under any possible bracket arrangements with entries on the letters $A_{i,n}$'s.

\begin{dfn}
{\rm We recursively define the sets $\mathcal{K}(n)_{k}$, $1\leq
k\leq n$, in the reverse order as follows:
\begin{enumerate}
\item[1)]  Let $\mathcal{K}(n)_n=\{A_{1,n},A_{2,n},\ldots,A_{n-1,n}\}.$
\item[2)] Suppose that $\mathcal{K}(n)_{k+1}$ is defined as a subset of Lie monomials on the letters
$$
A_{1,n},A_{2,n},\ldots,A_{n-1,n}
$$
with $k<n$. Let
$$\calA_k=\{W\in \mathcal{K}(n)_{k+1} \ | \ W \textrm{ does not contain } A_{k,n} \textrm{ in its entries}\}.$$
\item[3)] Define
$$\mathcal{K}(n)_k=\{W' \textrm{ and } [\cdots [[W', W_1],W_2],\ldots,W_t]\}
$$
for $W'\in \mathcal{K}(n)_{k+1}\smallsetminus \calA_k$ and $W_1,W_2,\ldots,W_t\in \calA_k$ with $t\geq 1$.
Note that $\mathcal{K}(n)_k$ is again a subset of  Lie monomials on letters $
A_{1,n},A_{2,n},\ldots,A_{n-1,n}$.
\end{enumerate}}
\end{dfn}

\begin{example}
{\rm Let $n=3$. The set $\mathcal{K}(3)_1$ is constructed by the following steps:
\begin{enumerate}
\item[1)] $\mathcal{K}(3)_3=\{A_{1,3},A_{2,3}\}$.
\item[2)] $\mathcal A_2=\{A_{1,2}\}$,
$$\mathcal{K}(3)_2=\{A_{2,3}, [[A_{2,3}, A_{1,3}],\ldots,A_{1,3}] \}.$$
\item[3)]
 $\mathcal A_1=\{A_{2,3}\}$,

 $\mathcal{K}(3)_1=\{[\cdots [A_{2,3}, A_{1,3}],\ldots, A_{1,3}], A_{2,3}], \ldots,  A_{2,3}]\}$.
 \end{enumerate}}
\end{example}

\begin{rem}
All elements of $\mathcal{K}(3)_1$ under the canonical inclusion
 $\mathrm{L^P(\Brun_n)}\hookrightarrow L(P_3)$
are mapped to the elements (not all) of a Hall basis for the free Lie subalgebra of
of $L(P_3)$ generated by $A_{1,3}$ and $A_{2,3}$.
\end{rem}
\begin{thm}\label{theorem8}
The Lie algebra $\LPBn$ is a free Lie algebra generated by $\mathcal{K}(n)_1$ as a set of free generators.
\end{thm}
\begin{proof}
The assertion follows from the statement that $\mathcal{K}(n)_k$ is a set of free generators for
$$
\Ker(d_n)\cap\Ker(d_{n-1})\cap \cdots\cap \Ker(d_k)
$$
for $1\leq k\leq n$. We prove this statement by induction in reverse order. For $k=n$ it follows from Proposition~\ref{proposition3}. Suppose that the statement holds for $k+1$ with $k<n$.  Let
$$\mathcal A_k=\{W\in \mathcal{K}(n)_{k+1} \ | \ W \textrm{ does not contain } A_{k,n} \textrm{ in its entries}\}$$
and let $\mathcal{B}_k=\mathcal{K}(n)_{k+1}\smallsetminus \mathcal{A}_k$. By induction,
$$
\Ker(d_n)\cap\Ker(d_{n-1})\cap \cdots\cap \Ker(d_{k+1})=L[\mathcal{K}(n)_{k+1}]
$$
is a free Lie algebra freely generated by $\mathcal{K}(n)_{k+1}$.
The Lie algebra generated by $\calA_k$ is a  Lie subalgebra of the Lie algebra
freely generated by $A_{1,n},\ldots,A_{n-1,n}$. Thus the Lie
homomorphism
$$
\phi\colon L[\calA_k]\longrightarrow L[A_{1,n},\ldots,A_{n-1,n}]
$$
with $\phi(W)=W$ for $W\in \calA_k$ is a monomorphism with its image given by the  Lie subalgebra generated by $\calA_k$.

Consider the homomorphism
$$
d_k\colon L[A_{1,n},\ldots, A_{n-1,n}]\longrightarrow L[A_{1,n-1},\ldots, A_{n-2,n-1}]
$$
given in formula~(\ref{face-dk}). We show that the composite
$$
d_k\circ \phi\colon L[\calA_k]\longrightarrow L[A_{1,n-1},\ldots, A_{n-2,n-1}].
$$
is a monomorphism. By the definition of $\calA_k$, the image
$\phi(L[\calA_k])$ is contained in the  Lie subalgebra
$$
L[A_{1,n},\ldots,A_{k-1,n},A_{k+1,n},\ldots,A_{n-1,n}]
$$
of $L[A_{1,n},\ldots,A_{n-1,n}]$. Thus there is a commutative diagram of Lie algebras
 \begin{diagram}
L[\calA_k]&\rTo^{\phi}& L[A_{1,n},\ldots, A_{n-1,n}]\\
\dTo>{\phi'}&\ruInto&\dTo>{d_k}\\
L[A_{1,n},\ldots,A_{k-1,n},A_{k+1,n},\ldots,A_{n-1,n}]& \rTo^{{d_k}|}&  L[A_{1,n-1},\ldots, A_{n-2,n-1}],\\
\end{diagram}
where $\phi'$ is defined by the same formula as $\phi$.
Since $\phi$ is a monomorphism, so is $\phi'$. From the definition, the restriction
$$
d_k|\colon L[A_{1,n},\ldots,A_{k-1,n},A_{k+1,n},\ldots,A_{n-1,n}]\longrightarrow L[A_{1,n-1},\ldots, A_{n-2,n-1}]
$$
is an isomorphism. It follows that $d_k\circ \phi\colon L[\calA_k]\to
L[A_{1,n-1},\ldots, A_{n-2,n-1}]$ is a monomorphism.

Observe that $d_k(W)=0$ for $W\in\mathcal{K}(n)_{k+1}\smallsetminus \calA_k$. There is a commutative diagram of Lie algebras
\begin{diagram}
L[\mathcal{K}(n)_{k+1}]&\rEq& \Ker(d_n)\cap\cdots\cap \Ker(d_{k+1})&\rInto & L[A_{1,n},\ldots, A_{n-1,n}]\\
\dTo>{\pi} &     &  &\rdTo>{d_k|} &\dTo>{d_k}\\
L[\calA_k]&\rInto^{d_k\circ\phi}&&& L[A_{1,n-1},\ldots, A_{n-2,n-1}],\\
\end{diagram}
where $\pi(W)=0$ for $W\in \mathcal{K}(n)_{k+1}\smallsetminus \calA_k$ and $\pi(W)=W$ for $W\in \calA_k$. It follows that
$$
\Ker(d_n)\cap\cdots\cap \Ker(d_{k})=$$ $$=\Ker(d_k|\colon \Ker(d_n)\cap\cdots\cap \Ker(d_{k+1})\to L[A_{1,n-1},\ldots, A_{n-2,n-1}])
$$
is given by the kernel of
$$
\pi\colon
L[\mathcal{K}(n)_{k+1}]=L[(\mathcal{K}(n)_{k+1}\smallsetminus
\calA_k)\sqcup\calA_k]\longrightarrow L[\calA_k],
$$
which is freely generated by $\mathcal{K}(n)_k$ by Lemma~\ref{lemma4}. This finishes the proof.
\end{proof}

\begin{example}
{\rm Let $n=4$. The set $\mathcal{K}(4)_1$ is constructed by the following steps:
\begin{enumerate}
\item[1)] $\mathcal{K}(4)=\{A_{1,4},A_{2,4},A_{3,4}\}$.
\item[2)] $\calA_3=\{A_{1,4},A_{2,4}\}$,
$$\mathcal{K}(4)_3=\{[[A_{3,4}, A_{j_1,4}],\ldots,A_{j_t,4}] \ | \ 1\leq j_1,\ldots,j_t\leq 2, \ t\geq 0\},$$ where, for $t=0$, $[[A_{3,4}, A_{j_1,4}],\ldots,A_{j_t,4}] =A_{3,4}$.
\item[3)]
For constructing $\mathcal{K}(4)_2$, let $W=[[A_{3,4}, A_{j_1,4}],\ldots,A_{j_t,4}]\in \mathcal{K}(4)_3$. If $W$ does not contain $A_{2,4}$, then $W=A_{3,4}$ or \newline
$W=[A_{3,4}, A_{j_1,4}],\ldots,A_{j_t,4}]$ with $j_1=j_2=\cdots=j_t=1$. Let
$$
\mathrm{ad}^t(b)(a)=[[a,b],b,\ldots,b]
$$
with $t$ entries of $b$, where $\mathrm{ad}^0(b)(a)=a$. Then $W$ does not contain $A_{2,4}$ if and only if
$$
W=\mathrm{ad}^t(A_{1,4})(A_{3,4})
$$
for $t\geq 0$. So $\calA_2=\{\mathrm{ad}^t(A_{1,4})(A_{3,4}), t\geq 0\}$.
 From the definition, $\mathcal{K}(4)_2$ is given by
$$
[[A_{3,4}, A_{j_1,4}],\ldots,A_{j_t,4}]\quad\textrm{ and}
$$
$$
[[[[A_{3,4}, A_{j_1,4}],\ldots,A_{j_t,4}], \mathrm{ad}^{s_1}(A_{1,4})(A_{3,4})],\ldots,
 \mathrm{ad}^{s_q}(A_{1,4})(A_{3,4})],
$$
where $1\leq j_1,\ldots,j_t\leq 2$ with at least one $j_i=2$, $s_1,\ldots,s_q\geq 0$ and
 $q\geq 1$.
\item[4)] For constructing $\mathcal{K}(4)_1$, let $W$ be an element of $\mathcal{K}(4)_2$,
$$
W=
$$
$$
\qquad \quad  [[[[A_{3,4}, A_{j_1,4}],\ldots,A_{j_t,4}], \mathrm{ad}^{s_1}(A_{1,4})(A_{3,4})],\ldots, \mathrm{ad}^{s_q}(A_{1,4})(A_{3,4})],
$$
where, for $q=0$, $W=[[A_{3,4}, A_{j_1,4}],\ldots,A_{j_t,4}]$.
Then $W$ does not contain $A_{1,4}$ if and only if $q=0$ and \newline
$W=[[A_{3,4}, A_{j_1,4}],\ldots,A_{j_t,4}]$ with $j_1=j_2=\cdots=j_t=2$, namely
$$
W=\mathrm{ad}^t(A_{2,4})(A_{3,4})
$$
for $t\geq 1$. So, $\calA_1=\{\mathrm{ad}^t(A_{2,4})(A_{3,4}), t\geq 1\}$.
Thus $\mathcal{K}(4)_1$, which is a set of free generators for $\mathrm{L}^P(\mathrm{Brun}_4)$, is given by
$$
W \textrm{ and } [[W, \mathrm{ad}^{l_1}(A_{2,4})(A_{3,4})],\ldots, \mathrm{ad}^{l_p}(A_{2,4})(A_{3,4})],
$$
where $l_i\geq 1$ for $1\leq i\leq p$ with $p\geq 1$ and
$$
W=
$$
$$
\qquad \quad
[[[[A_{3,4}, A_{j_1,4}],\ldots,A_{j_t,4}], \mathrm{ad}^{s_1}(A_{1,4})(A_{3,4})],\ldots, \mathrm{ad}^{s_q}(A_{1,4})(A_{3,4})]
$$
is an element of $\mathcal{K}(4)_2$,
so that each of $A_{2,4}$ and $A_{1,4}$ appears in $W$ at least once. \hfill $\Box$
\end{enumerate}
}
\end{example}

From the above example, one can see that the set $\mathcal{K}(n)_1$ is still complicated in the sense that its elements involve the iterated operations of  normal Lie brackets from left to right $[\cdots[\ , \ ],\ldots, \ ]$.

\begin{question}
Determine a set of free generators for $\LPBn$ using normal Lie brackets from left to right.
\end{question}

\section{The symmetric Lie products of Lie ideals}

Let $L$ be a Lie algebra and $I_1, \ldots, I_n$ its ideals. We
define the notion of the fat bracket sum and the symmetric bracket
sum of ideals which is similar to the corresponding fat commutator
product and symmetric commutator product in groups
 \cite{BMVW}, \cite{LiWu}.
Given a Lie algebra $L$, and a set of its ideals  $I_1,\ldots, I_n,\
(l\geq 2)$, the fat bracket sum of these ideals is defined to be the
Lie ideal of $L$ generated by all of the commutators
\begin{equation}\label{equation1.1}
\beta^t(a_{i_1},\ldots,a_{i_t}),
\end{equation}
where
\begin{enumerate}
\item[1)] $1\leq i_s\leq n$;
\item[2)] $\{i_1,\ldots,i_t\}=\{1,\ldots,n\}$, so, each integer in $\{1,2,\cdots,n\}$ appears as at least one of the
integers $i_s$;
\item[3)] $a_j\in I_j$;
\item[4)] $\beta^t$ runs over all of the  bracket arrangements of weight $t$ (with $t\geq n$).
\end{enumerate}
 The symmetric bracket sum of
these ideals is defined as
$$
[[I_1,I_2],\ldots, I_l]_S:=\sum_{\sigma\in
\Sigma_l}[[I_{\sigma(1)},I_{\sigma(2)}],\ldots,I_{\sigma(n)}],
$$
where $\Sigma_n$ is the symmetric group o $n$ letters.

As in \cite{BMVW}, \cite{LiWu} we can prove that the symmetric
bracket sum of $I_1,\ldots, I_n,\ (n\geq 2)$ is the same as the fat
bracket sum.
\begin{thm}\label{theorem3.1}
Let $I_j$ be any Lie ideals of a Lie algebra $L$ with $1\leq j\leq
n$. Then
$$
[[I_1,I_2,\ldots,I_n]]=[[I_1,I_2],\ldots,I_n]_S.
$$
\end{thm}

To prove this theorem, we need the following lemmas. 

\begin{lem}\label{lemma3.1}
Let $L$ be a Lie algebra and let $A,B,C$ be Lie ideals of $L$. Then
any one of the Lie ideal $[A,[B,C]]$, $[[A,B],C]$ and $[[A,C],B]$ is
a Lie ideal of the sum of the other two.\hfill $\Box$
\end{lem}

\begin{lem}\label{lemma3.2}
Let $I_1,\ldots,I_n$ be Lie ideals of $L$. Let $a_j\in I_j$ for
$1\leq j\leq n$. Then
$$
\beta^n(a_{\sigma(1)},\ldots,a_{\sigma(n)})\in
[[I_1,I_2],\ldots,I_n]_S
$$
for any $\sigma\in \Sigma_n$ and any bracket arrangement $\beta^n$
of weight $n$.
\end{lem}
\begin{proof}
The proof is given by double induction. The first induction is on
$n$. Clearly the assertion holds for $n=1$. Suppose that the
assertion holds for $m$ with $m<n$. Given an element
$\beta^n(a_{\sigma(1)},\ldots a_{\sigma(n)})$ as in the statement of the
lemma we have
$$
\beta^n(a_{\sigma(1)},\ldots,a_{\sigma(n)})=[\beta^p(a_{\sigma(1)},\ldots,a_{\sigma(p)}),\beta^{n-p}(a_{\sigma(p+1)},\ldots,a_{\sigma(n)})]
$$
for some bracket arrangements $\beta^p$ and $\beta^{n-p}$ with
$1\leq p\leq n-1$. The second induction is on $q=n-p$. If $q=1$, we
have
$$
\beta^{n-1}(a_{\sigma(1)},\ldots,a_{\sigma(n-1)})\in
[[I_{\sigma(1)},I_{\sigma(2)}],\ldots,I_{\sigma(n-1)}]_S
$$
by the first induction and so
$$
\begin{array}{rcl}
\beta^n(a_{\sigma(1)},\ldots,a_{\sigma(n)})&=&[\beta^{n-1}(a_{\sigma(1)},\ldots,a_{\sigma(n-1)}),a_{\sigma(n)}]\\
&\in& [[[I_{\sigma(1)},I_{\sigma(2)}],\ldots,I_{\sigma(n-1)}]_S, I_{\sigma(n)}]\\
\end{array}
$$
with
$$
\begin{array}{rl}
& [[[I_{\sigma(1)},I_{\sigma(2)}],\ldots,I_{\sigma(n-1)}]_S, I_{\sigma(n)}]\\
=&\left[\sum_{\tau\in\Sigma_{n-1}}[[I_{\tau(\sigma(1))},I_{\tau(\sigma(2))}],\ldots,I_{\tau(\sigma(n-1))}],I_{\sigma(n)}\right]\\
=&\sum_{\tau\in\Sigma_{n-1}}[[[I_{\tau(\sigma(1))},I_{\tau(\sigma(2))}],\ldots,I_{\tau(\sigma(n-1))}], I_{\sigma(n)}]\\
\leq& [[I_1,I_2],\ldots,I_n]_S.\\
\end{array}
$$
Now suppose that the assertion holds for $q'=n-p<q$. By the first
induction, we have
$$
\beta^p(a_{\sigma(1)},\ldots,a_{\sigma(p)})\in
[[I_{\sigma(1)},I_{\sigma(2)}],\ldots,I_{\sigma(p)}]_S
$$
and
$$
\beta^{n-p}(a_{\sigma(p+1)},\ldots,a_{\sigma(n)})\in
[[I_{\sigma(p+1)},I_{\sigma(p+2)}],\ldots,I_{\sigma(n)}]_S.
$$
Thus
$$
\beta^{n}(a_{\sigma(1)},\ldots,a_{\sigma(n)})\in
\left[[[I_{\sigma(1)},I_{\sigma(2)}],\ldots,I_{\sigma(p)}]_S,
[[I_{\sigma(p+1)},I_{\sigma(p+2)}],\ldots,I_{\sigma(n)}]_S\right].
$$
Then
$$\beta^{n}(a_{\sigma(1)},\ldots,a_{\sigma(n)})\in\sum\limits_{
\begin{array}{c}
\tau\in\Sigma_p\\
\rho\in\Sigma_{n-p}\\
\end{array}}
\left[[[I_{\tau(\sigma(1))},\ldots,I_{\tau(\sigma(p))}],
[[I_{\rho(\sigma(p+1))},\ldots,I_{\rho(\sigma(n))}]\right],
$$
where $\Sigma_{n-p}$ acts on $\{\sigma(p+1),\ldots,\sigma(n)\}$. By
applying Jacobi identity, we have
$$
\begin{array}{rl}
&\left[[[I_{\tau(\sigma(1))},\ldots,I_{\tau(\sigma(p))}], [[I_{\rho(\sigma(p+1))},\ldots,I_{\rho(\sigma(n))}]\right]\\
=&\left[[[I_{\tau(\sigma(1))},\ldots,I_{\tau(\sigma(p))}], \left[[[I_{\rho(\sigma(p+1))},\ldots, I_{\rho(\sigma(n-1))}],I_{\rho(\sigma(n))}\right]\right]\\
\leq &\left[\left[[[I_{\tau(\sigma(1))},\ldots,I_{\tau(\sigma(p))}], [[I_{\rho(\sigma(p+1))},\ldots, I_{\rho(\sigma(n-1))}]\right],I_{\rho(\sigma(n))}\right]\\
&+\left[ \left[[[I_{\tau(\sigma(1))},\ldots,I_{\tau(\sigma(p))}],I_{\rho(\sigma(n))}\right], [[I_{\rho(\sigma(p+1))},\ldots, I_{\rho(\sigma(n-1))}]\right].\\
\end{array}
$$
Note that
$A=\left[\left[[[I_{\tau(\sigma(1))},\ldots,I_{\tau(\sigma(p))}],
[[I_{\rho(\sigma(p+1))},\ldots,
I_{\rho(\sigma(n-1))}]\right],I_{\rho(\sigma(n))}\right]$ is
generated by the elements of the form
$$
\left[\left[[[a'_{\tau(\sigma(1))},\ldots,a'_{\tau(\sigma(p))}],
[[a'_{\rho(\sigma(p+1))},\ldots,
a'_{\rho(\sigma(n-1))}]\right],a'_{\rho(\sigma(n))}\right]
$$
with $a'_j\in I_j$. By the second induction in case when $q=1$, the
above elements lie in $[[I_1,I_2],\ldots,I_n]_S$ and so
$$
A\leq [[I_1,I_2],\ldots,I_n]_S.
$$
Similarly, by the second induction hypothesis, $$\left[
\left[[[I_{\tau(\sigma(1))},\ldots,I_{\tau(\sigma(p))}],I_{\rho(\sigma(n))}\right],
[[I_{\rho(\sigma(p+1))},\ldots, I_{\rho(\sigma(n-1))}]\right]$$ is a
Lie ideal of $[[I_1,I_2],\ldots,I_n]_S$. It follows that
$$
T\leq [[I_1,I_2],\ldots,I_n]_S
$$
and so
$$
\beta^{n}(a_{\sigma(1)},\ldots,a_{\sigma(n)})\in
[[I_1,I_2],\ldots,I_n]_S.
$$
Both the first and second inductions are finished, hence the result.
\end{proof}

\begin{lem}\label{lemma3.3}
Let $L$ be a Lie algebra and let $I_1,\ldots,I_n$ be Lie ideals of
$L$. Let $(i_1,i_2,\ldots,i_p)$ be a sequence of integers with
$1\leq i_s\leq n$. Suppose that
$$\{i_1,i_2,\ldots,i_p\}=\{1,2,\ldots,n\}.$$ Then
$$
[[I_{i_1},I_{i_2}],\ldots,I_{i_p}]\leq [[I_1,I_2],\ldots,I_n]_S.
$$
\end{lem}
\begin{proof} We also apply double induction. The first induction is on  $n$.
 The assertion clearly holds for $n=1$. Suppose that the assertion holds for $n-1$ with $n>1$.
From the condition $\{i_1,i_2,\ldots,i_p\}=\{1,2,\ldots,n\}$,
we have $p\geq n$. When $p=n$, $(i_1,\ldots,i_n)$ is a permutation
of $(1,\ldots,n)$ and so $$[[I_{i_1},I_{i_2}],\ldots,I_{i_n}]\leq
[[I_1,I_2],\ldots,I_n]_S.$$ Suppose that
$$
[[I_{j_1},I_{j_2}],\ldots,I_{j_q}]\leq [[I_1,I_2],\ldots,I_n]_S
$$
for any sequence $(j_1,\ldots,j_q)$ with $q<p$ and
$\{j_1,\ldots,j_q\}=\{1,\ldots,n\}$. Let $(i_1,\ldots,i_p)$ be a
sequence with $\{i_1,\ldots,i_p\}=\{1,\ldots,n\}$.
If $i_p\in \{i_1,\ldots,i_{p-1}\}$, then
$\{i_1,\ldots,i_{p-1}\}=\{1,\ldots,n\}$ and so
$$
[[I_{i_1},I_{i_2}],\ldots,I_{i_{p-1}}]\leq [[I_1,I_2],\ldots,I_n]_S
$$
by the second induction hypothesis. It follows that
$$
[[I_{i_1},I_{i_2}],\ldots,I_{i_p}]\leq [[I_1,I_2],\ldots,I_n]_S.
$$
If $i_p\not\in \{i_1,\ldots,i_{p-1}\}$ , we may assume that $i_p=n$.
Then $$\{i_1,\ldots,i_{p-1}\}=\{1,\ldots,n-1\}$$ and so
$$
[[I_{i_1},I_{i_2}],\ldots,I_{i_{p-1}}]\leq
[[I_1,I_2],\ldots,I_{n-1}]_S
$$
by the first induction hypothesis. From Lemma~\ref{lemma3.2}, we
have
$$
[[I_{i_1},I_{i_2}],\ldots,I_{i_p}]\leq [[I_1,I_2],\ldots,I_n]_S.
$$
The inductions are  finished, hence the result holds.
\end{proof}

\begin{lem}\label{lemma3.4}
Let $L$ be a Lie algebra and let $I_1,\ldots,I_n$ be Lie ideal of
$L$ with $n\geq 2$. Let $(i_1,\ldots,i_p)$ and $(j_1,\ldots,j_q)$ be
sequences of integers such that
$\{i_1,\ldots,i_p\}\cup\{j_1,\ldots,j_q\}=\{1,2,\ldots,n\}$. Then
$$
[[[I_{i_1},I_{i_2}],\ldots,I_{i_p}],[[I_{j_1},I_{j_2}],\ldots,I_{j_q}]]\leq
[[I_1,I_2],\ldots,I_n]_S.
$$
\end{lem}
\begin{proof}
Again we use the double induction on $n$ and $q$ with $n\geq 2$
and $q\geq 1$. First we prove that the assertion holds for $n=2$. If
$\{i_1,\ldots,i_p\}=\{1,2\}$ or $\{j_1,\ldots,j_q\}=\{1,2\}$, we
have
$$
[[I_{i_1},I_{i_2}],\ldots,I_{i_p}]\leq [[I_1,I_2]_S \textrm{ or
}[[I_{j_1},I_{j_2}],\ldots,I_{j_q}]\leq [[I_1,I_2]_S$$ by
Lemma~\ref{lemma3.3}  and so  $$
[[[I_{i_1},I_{i_2}],\ldots,I_{i_p}],[[I_{j_1},I_{j_2}],\ldots,I_{j_q}]]\leq
[[I_1,I_2]_S.
$$
Otherwise, $i_1=\cdots=i_p$ and $j_1=\cdots=j_q $, since
$\{i_1,\ldots,i_p\}\cup\{j_1,\ldots,j_q\}=\{1,2\}$, we may assume
that  $i_1=\cdots=i_p=1, j_1=\cdots=j_q =2$, then
$$[[I_{i_1},I_{i_2}],\ldots,I_{i_p}]\leq I_1 \textrm{ and
}[[I_{j_1},I_{j_2}],\ldots,I_{j_q}]\leq I_2$$ and so
 $$
[[[I_{i_1},I_{i_2}],\ldots,I_{i_p}],[[I_{j_1},I_{j_2}],\ldots,I_{j_q}]]\leq
[[I_1,I_2]_S.
$$
Suppose the assertion holds for $n-1$, that is $$
[[[I_{i_1},I_{i_2}],\ldots,I_{i_p}],[[I_{j_1},I_{j_2}],\ldots,I_{j_q}]]\leq
[[I_1,I_2],\ldots,I_{n-1}]_S.
$$ when $\{i_1,\ldots,i_p\}\cup\{j_1,\ldots,j_q\}=\{1,2,\ldots,n-1\}$. We shall use the second induction on $q$ to prove that the assertion holds for $n$.
If $q=1$, the assertion follows by Lemma~\ref{lemma3.3}. Suppose
that the assertion holds for $q-1$. By Lemma~\ref{lemma3.1},
$[[[I_{i_1},I_{i_2}],\ldots,I_{i_p}],[[I_{j_1},I_{j_2}],\ldots,I_{j_q}]]$
is a Lie ideal of the sum
$$
[[[[I_{i_1},\ldots,I_{i_p}],[[I_{j_1},\ldots,I_{j_{q-1}}]],I_{j_q}]+
[[[[I_{i_1},\ldots,I_{i_p}],I_{j_q}],[[I_{j_1},\ldots,I_{j_{q-1}}]].
$$
By the second induction we have $$
[[[[I_{i_1},\ldots,I_{i_p}],I_{j_q}],[[I_{j_1},\ldots,I_{j_{q-1}}]]\leq
[[I_1,I_2],\ldots,I_n]_S.$$
If $\{i_1,\ldots,i_p\}\cup\{j_1,\ldots,j_{q-1}\}=\{1,2,\ldots,n\}$,
by the second induction
$$[[[I_{i_1},\ldots,I_{i_p}],[[I_{j_1},\ldots,I_{j_{q-1}}]]\leq
[[I_1,I_2],\ldots,I_n]_S$$ and hence
$$[[[[I_{i_1},\ldots,I_{i_p}],[[I_{j_1},\ldots,I_{j_{q-1}}]],I_{j_q}]\leq
[[I_1,I_2],\ldots,I_n]_S. $$
If $\{i_1,\ldots,i_p\}\cup\{j_1,\ldots,j_{q-1}\}\neq
\{1,2,\ldots,n\}$, we may assume that
$$\{i_1,\ldots,i_p\}\cup\{j_1,\ldots,j_{q-1}\}=\{1,2,\ldots,n-1\}$$
and $j_q=n$. By the first induction,
$$[[[I_{i_1},\ldots,I_{i_p}],[[I_{j_1},\ldots,I_{j_{q-1}}]]\leq
[[I_1,I_2],\ldots,I_{n-1}]_S.$$ Then
$$[[[[I_{i_1},\ldots,I_{i_p}],[[I_{j_1},\ldots,I_{j_{q-1}}]],I_{j_q}]\leq
[[I_1,I_2],\ldots,I_n]_S.$$
It follows that
$[[[I_{i_1},I_{i_2}],\ldots,I_{i_p}],[[I_{j_1},I_{j_2}],\ldots,I_{j_q}]]\leq
[[I_1,I_2],\ldots,I_n]_S$. The double induction is finished, hence
the result.
\end{proof}

{\it Proof of Theorem~\ref{theorem3.1}.} Clearly
$[[I_1,I_2],\ldots,I_n]_S\leq [[I_1,I_2,\ldots, I_n]]$. We prove by
induction on $n$ that $$[[I_1,I_2,\ldots,I_n]]\leq
[[I_1,I_2],\ldots,I_n]_S.$$ The assertion holds for $n=1$.
We now make the first induction hypothesi that for all
$1\leq s<n$ and for any Lie ideals $I_1,\ldots,I_s$ of $L$
\begin{equation}
[[I_1,I_2,\ldots, I_s]]\leq [[I_1,I_2],\ldots, I_s]_S.
\label{induction1}
\end{equation}
Let $I_1,\ldots, I_n$ be arbitrary Lie ideals of $L$. By definition,
$[[I_1,I_2,\ldots,I_n]]$ is generated by all commutators
$$\beta^t(a_{i_1},\ldots,a_{i_t})$$ of weight $t$ such that
$\{i_1,i_2,\ldots,i_t\}=\{1,2,\ldots,n\}$ with $a_j\in I_j$. To
prove that each generator $\beta^t(a_{i_1},\ldots,a_{i_t})\in
[[I_1,I_2],\ldots,I_n]_S$, we start the second induction on the
weight $t$ of $\beta^t$ with $t\geq n$. If $t=n$, then
$(i_1,\ldots,i_n)$ is a permutation of $(1,\ldots,n)$ and so the
assertion holds by Lemma~\ref{lemma3.2}.
Let $n\leq k$ and let
$$
\beta^k(a'_{i_1},\ldots,a'_{i_k})
$$
be any bracket arrangement of weight $k$ such that
\begin{enumerate}
\item[1)] $1\leq i_s\leq n$;
\item[2)] $\{i_1,\ldots,i_k\}=\{1,\ldots,n\}$;
\item[3)] $a'_j\in I_j$;
\end{enumerate}
 Now assume that the second hypothesis holds:
\begin{equation}
\beta^k(a'_{i_1},\ldots,a'_{i_k})\in
[[I_1,I_2],\ldots,I_n]_S
\label{induction2}
\end{equation}
for all $k$ such that $n\leq k<t$.

Let $\beta^t(a_{i_1},\ldots,a_{i_t})$ be any bracket arrangement of
weight $t$ with $\{i_1,\ldots,i_t\}=\{1,\ldots,n\}$ and $a_j\in I_j$
for $1\leq j\leq n$. From the definition of bracket arrangement, we
have
$$
\beta^t(a_{i_1},\ldots,a_{i_t})=[\beta^p(a_{i_1},\ldots,a_{i_p}),
\beta^{t-p}(a_{i_{p+1}},\ldots,a_{i_t})]
$$
for some bracket arrangements $\beta^p$ and $\beta^{t-p}$ of weight
$p$ and $t-p$, respectively, with $1\leq p\leq n-1$. Let
$$
A=\{i_1,\ldots,i_p\}\textrm{ and } B=\{i_{p+1},\ldots,i_t\}.
$$
Then both $A$ and $B$ are the subsets of $\{1,\ldots,n\}$ with
$A\cup B=\{1,\ldots,n\}$.

Suppose that the cardinality $|A|=n$ or $|B|=n$. We may assume that
$|A|=n$. By hypothesis~\ref{induction2},
$$
\beta^p(a_{i_1},\ldots,a_{i_p})\in [[I_1,I_2],\ldots,I_n]_S.
$$
Since $[[I_1,I_2],\ldots,I_n]_S$ is a Lie ideal of $L$, we have
$$
\beta^t(a_{i_1},\ldots,a_{i_t})=[\beta^p(a_{i_1},\ldots,a_{i_p}),
\beta^{t-p}(a_{i_{p+1}},\ldots,a_{i_t})]\in
[[I_1,I_2],\ldots,I_n]_S.
$$
This proves the result in this case.

Suppose that $|A|<n$ and $|B|<n$. Let $A=\{l_1,\ldots,l_a\}$ with
$1\leq l_1<l_2<\cdots<l_a\leq n$ and $1\leq a<n$, and let
$B=\{k_1,\ldots,k_b\}$ with $1\leq k_1<k_2<\cdots< k_b$ and $1\leq
b<n$. Observe that
$$
\beta^p(a_{i_1},\ldots,a_{i_p})\in
[[I_{l_1},I_{l_2},\ldots,I_{l_a}]].
$$
By hypothesis~\ref{induction1},
$$
[[I_{l_1},I_{l_2},\ldots,I_{l_a}]]=[[I_{l_1},I_{l_2}],\ldots,I_{l_a}]_S.
$$
Thus
$$
\beta^p(a_{i_1},\ldots,a_{i_p})\in
[[I_{l_1},I_{l_2}],\ldots,I_{l_a}]_S.
$$
Similarly
$$
\beta^{t-p}(a_{i_{p+1}},\ldots,a_{i_t})\in
[[I_{k_1},I_{k_2}],\ldots,I_{k_b}]_S.
$$
It follows that
$$\beta^t(a_{i_1},\ldots,a_{i_t})\in\left[[[I_{l_1},I_{l_2}],\ldots,I_{l_a}]_S,
[[I_{k_1},I_{k_2}],\ldots,I_{k_b}]_S\right].
$$
From Lemma~\ref{lemma3.4}, we have
$$
\left[[[I_{l_{\sigma(1)}},I_{l_{\sigma(2)}}],\ldots,I_{l_{\sigma(a)}}],
[[I_{k_{\tau(1)}},I_{k_{\tau(2)}}],\ldots,I_{k_{\tau(b)}}]\right]\leq
[[I_1,I_2],\ldots,I_n]_S
$$
for all $\sigma\in\Sigma_a$ and $\tau\in\Sigma_b$ because
$\{l_1,\ldots,l_a\}\cup \{k_1,\ldots,k_b\}=A\cup
B=\{1,2,\ldots,n\}$. It follows that
$$
\left[[[I_{l_1},I_{l_2}],\ldots,I_{l_a}]_S,
[[I_{k_1},I_{k_2}],\ldots,I_{k_b}]_S\right]\leq
[[I_1,I_2],\ldots,I_n]_S.
$$
Thus
$$
\beta^t(a_{i_1},\ldots,a_{i_t})\in [[I_1,I_2],\ldots,I_n]_S.
$$
The inductions are finished, hence Theorem~\ref{theorem3.1}.

Let us denote the ideal
$$L[A_{k,n},  [\cdots [A_{k,n}, A_{j_1,n}], \ldots, A_{j_m, n} ] \ | \
j_i\not= k, n; \ j_i\le n-2, i\le m; \ m\ge 1]
$$
by $I_k$. Then we have the following theorem.
\begin{thm}\label{proposition3.6} The Lie subalgebra
$\LPBn$ and the symmetric bracket sum $[[I_1,I_2],\ldots, I_{n-1}]_S$ are
equal as subalgebras in  $L(P_n)$:
\begin{equation*}
\LPBn = [[I_1,I_2],\ldots, I_{n-1}]_S.
\end{equation*}
\end{thm}
\begin{proof}
It is evident that the symmetric bracket sum $[[I_1,I_2],\ldots,
I_{n-1}]_S$ lies in the kernels of all $d_i$. On the other hand,
from Theorem~\ref{theorem8}, $\LPBn$ is given as ``\textit{fat Lie
product}'' of $I_1,\ldots, I_{n-1}$ because each element in
$\mathcal{K}(n)_1$ is a Lie monomial containing each of
$A_{1,n},\ldots,A_{n-1,n}$. We know that $\mathcal{K}(n)_1\subseteq
[[I_1,\ldots, I_{n-1}]]=[[I_1,I_2],\ldots, I_{n-1}]_S$. Thus $\LPBn$ is
contained in the symmetric bracket sum $[[I_1,I_2],\ldots, I_{n-1}]_S$.
\end{proof}


\section{The Rank of $L^P_q(\Brun_n)$}
Observe that the Lie algebra $L(P)$ is of finite type in the sense
that each homogeneous component $L_k(P_n)$ is a free abelian group
of finite rank. Thus the subgroup
$$
\LPBn\cap L_k(P_n)
$$
is a free abelian group of finite rank. The purpose of this section
to give a formula on the rank of $L^P_q(\Brun_n)$
\subsection{A decomposition formula on bi-$\Delta$-groups}

By the definition of bi-$\Delta$-groups and the
face and co-face operation on $\mathbb{P}=\{\mathbb{P}_{n}\}_{n\geq
0}$, we have the following lemma.
\begin{lem}\label{lemma 3.4}
 For every $q\geq 0$, $L_q(\mathbb{P})=\{L_q(\mathbb{P}_n)\}_{n\geq 0}$ is a
bi-$\Delta$-group.\hfill $\Box$
\end{lem}

Let $\mathcal{G}=\{G_n\}_{n\geq0}$ be a bi-$\Delta$-group.  Define
$$
\mathcal{Z}_n(\mathcal{G})=\bigcap_{i=0}^n\mathrm{Ker}(d_i\colon G_n\to G_{n-1}).
$$
The following statement on bi-$\Delta$-groups is proved in~\cite[Proposition 1.2.9]{Wu4}.
\begin{thm}[Decomposition Theorem of bi-$\Delta$-groups]
Let $\calG=\{G_n\}_{n\geq0}$ be a bi-$\Delta$-group. Then $G_n$ is
the (iterated) semi-direct product the subgroups
$$
d^{i_k}d^{i_{k-1}}\cdots d^{i_1} (\calZ_{n-k}(\calG)),
$$
$0\leq i_1<\cdots<i_k\leq n$, $0\leq k\leq n$, with lexicographic
from right.\hfill $\Box$
\end{thm}
\begin{cor}\label{Corollary4.3}
Let $\mathcal G=\{G_n\}_{n\geq0}$ be a bi-$\Delta$-group such that each $G_n$ is an abelian group. Then there is direct sum decomposition
$$
G_n=\bigoplus_{
\begin{array}{c}
0\leq i_1<\cdots<i_k\leq n\\
0\leq k\leq n\\
\end{array}}
d^{i_k}d^{i_{k-1}}\cdots d^{i_1} (\calZ_{n-k}(\calG))
$$
for each $n$.\hfill $\Box$
\end{cor}

\subsection{The Rank of $L^P_q(\Brun_n)$}

Let $\mathcal{G}=L_q(\mathbb{P})$. Then
$\mathcal{Z}_n(L_q(\mathbb{P})=L_q^P(\Brun_{n+1})$ by Proposition~\ref{Proposition2}.  Let
$d^i=\mathbbm{d}^{i-1}:\mathbb{P}_{n-1}=P_{n}\rightarrow \mathbb{P}_{n}=P_{n+1}$ is obtained
by adding a trivial $i$st string in front of the other strings
$(i=1,2,\cdots, n)$. By Corollary~\ref{Corollary4.3}, we have the following decomposition.

\begin{prop}
There is a decomposition
$$
L_q(P_{n})=\bigoplus_{
\begin{array}{c}
1\leq i_1<\cdots<i_k\leq n\\
0\leq k\leq n-1\\
\end{array}}
d^{i_k}d^{i_{k-1}}\cdots d^{i_1} (L^P_q(\Brun_{n-k}))
$$
for each $n$ and $q$. \hfill $\Box$
\end{prop}

\begin{cor}\label{corollary4.5}
There is a formula
$$
\mathrm{rank}(L_q(P_n))=\sum_{k=0}^{n-1} \binom{n}{k} \mathrm{rank}(L^P_q(\Brun_{n-k}))
$$
for each $n$ and $q$. \hfill $\Box$
\end{cor}

\begin{thm}
$$
\mathrm{rank}(L_q^P(\Brun_n))=\sum_{k=0}^{n-1}(-1)^k \binom{n}{k} \mathrm{rank}(L_q(P_{n-k}))
$$
for each $n$ and $q$, where $P_1=0$ and, for $m\geq2$,
$$
\mathrm{rank}(L_q(P_m))=\frac{1}{q}\sum_{k=1}^{m-1}\sum_{d|q}\mu(d)k^{q/d}
$$
with $\mu$ the M\"obis function.
\end{thm}
\begin{proof}
From the semi-direct product decomposition of Lie algebras,
$$
L(P_m)\cong L(P_{m-1})\oplus L(F_{m-1}),
$$
we have
$$
\mathrm{rank}(L_q(P_m))=\sum_{k=1}^{m-1}\mathrm{rank}(L_q(F_k))$$
for $m\geq 2$. Since $L(F_k)$ is the free Lie algebra on a set of $k$-elements,
$$
\mathrm{rank}(L_q(F_k))=\frac{1}{q}\sum_{d|q}\mu(d)k^{q/d}
$$
and so
$$
\mathrm{rank}(L_q(P_m))=\frac{1}{q}\sum_{k=1}^{m-1}\sum_{d|q}\mu(d)k^{q/d}.
$$
                                                                                                                                                                                                                                                                                                                                                                                                                                                                                                                                                                                                                                                                                                                                                                                                                                                                                                                                                                                                                                                                                                                                                                                                                                                                                                                                                                                                                                                                                          Now let $b_q(P_n)=\mathrm{rank}(L_q(P_n))$ and $b_q^P(\Brun_n)=\mathrm{rank}(L^P_q(\Brun_{n-k}))$.
By Corollary~\ref{corollary4.5}, we have
$$
\left(
\begin{array}{c}
b_q(P_n)\\
b_q(P_{n-1})\\
b_q(P_{n-2})\\
\vdots\\
b_q(P_1)\\
\end{array}
\right)=
\left(
\begin{array}{ccccc}
1&\binom{n}{1}&\binom{n}{2}&\cdots&\binom{n}{n-1}\\
0&1&\binom{n-1}{1}&\cdots&\binom{n-1}{n-2}\\
0&0&1&\cdots&\binom{n-2}{n-3}\\
\vdots&\vdots&\vdots&\cdots&\vdots\\
0&0&0&\cdots&1\\
\end{array}
\right)
\left(
\begin{array}{c}
b_q^P(\Brun_n)\\
b_q^P(\Brun_{n-1})\\
b_q^P(\Brun_{n-2})\\
\vdots\\
b_q^P(\Brun_1)\\
\end{array}
\right).
$$
Let
$A_n$ be the coefficient matrix of the above linear equations. Then
$$
A_n^{-1}=
\left(
\begin{array}{cccccc}
1&-\binom{n}{1}&\binom{n}{2}&-\binom{n}{3}&\cdots&(-1)^{n-1}\binom{n}{n-1}\\
0&1&-\binom{n-1}{1}&\binom{n-1}{2}&\cdots&(-1)^{n-2}\binom{n-1}{n-2}\\
0&0&1&-\binom{n-2}{1}&\cdots&(-1)^{n-3}\binom{n-2}{n-3}\\
0&0&0&1&\cdots&(-1)^{n-4}\binom{n-3}{n-4}\\
\vdots&\vdots&\vdots&\vdots&\cdots&\vdots\\
0&0&0&0&\cdots&1\\
\end{array}
\right)
$$
and hence the result.
\end{proof}

\bigskip

\section{Acknowledgments}
The first author is supported by NSFC (11201314) of China and the Excellent Young Scientist Fund of Shijiazhuang Tiedao University.
The second author is partially
supported by the Laboratory of Quantum Topology of Chelyabinsk State University
(Russian Federation government grant 14.Z50.31.0020) and
 RFBR grants  14-01-00014 and
13-01-92697-IND.
 The last author is
partially supported by the Singapore Ministry of Education research
grant (AcRF Tier 1 WBS No. R-146-000-190-112) and a grant (No.
11329101) of NSFC of China.

The second author is thankful to L.A.Bokut for the discussions on the subject of the paper.

\end{document}